\documentclass[12pt, a4paper]{amsart}
\usepackage{enumerate, longtable}
\usepackage{amsmath, amscd, amsfonts, amsthm, amssymb, latexsym, comment, stmaryrd, graphicx}
\usepackage{xcolor}
\usepackage[all]{xy}
\usepackage{a4wide}
\usepackage[
        colorlinks,
        backref,
]{hyperref}

\theoremstyle{plain}
\newtheorem{theorem}{Theorem}[section]
\newtheorem{proposition}[theorem]{Proposition}
\newtheorem{lemma}[theorem]{Lemma}

\theoremstyle{definition}
\newtheorem{definition}[theorem]{Definition}

\newcommand{\Gal}{{\rm Gal}}

\newcommand{\ZZ}{\mathbb{Z}}

\newcommand{\FF}{\mathbb{F}}

\newcommand{\M}{\mathcal{M}}

\def\FF{\mathbb{F}}

\def\AA{\mathbb{A}}

\def\AA{\mathbb{A}}

\def\lra{\longrightarrow}

\newcommand{\bfrho}{\mbox{\boldmath$\rho$}}
\newcommand{\bfomega}{\mbox{\boldmath$\omega$}}

\begin{document}

\title[Sums of two squares over finite fields]{Sums of two squares in short intervals in polynomial rings over finite fields}
\author{Efrat Bank}
\address{Raymond and Beverly Sackler School of Mathematical Sciences, Tel Aviv University, Tel Aviv 69978, Israel}
\email{bankefrat@gmail.com}
\author{Lior Bary-Soroker}
\address{Raymond and Beverly Sackler School of Mathematical Sciences, Tel Aviv University, Tel Aviv 69978, Israel}
\email{barylior@post.tau.ac.il}
\author{Arno Fehm}
\address{Universit\"at Konstanz, Fachbereich Mathematik und Statistik, Fach D 203, 78457 Konstanz, Germany}
\email{arno.fehm@uni-konstanz.de}

\maketitle

\begin{abstract}
Landau's theorem asserts that the asymptotic density of sums of two squares in the interval $1\leq n\leq x$ is $K/{\sqrt{\log x}}$, where $K$ is the Landau-Ramanujan constant.
It is an old problem in number theory whether the asymptotic density remains the same in intervals $|n-x|\leq x^{\epsilon}$ for a fixed $\epsilon$ and $x\to \infty$.
 
This work resolves a function field analogue of this problem, in the limit of a large finite field. More precisely, consider monic $f_0\in \mathbb{F}_q[T]$ of degree $n$ and take $\epsilon$ with $1>\epsilon\geq \frac2n$. Then the asymptotic density of polynomials $f$ in the `interval' $\deg(f-f_0)\leq \epsilon n$ that are of the form $f=A^2+TB^2$, $A,B\in \FF_q[T]$ is $\frac{1}{4^n}\binom{2n}{n}$ as $q\to \infty$. This density agrees with the  asymptotic density of such monic $f$'s of degree $n$ as $q\to \infty$, as was shown by the second author, Smilanski, and Wolf.

A key point in the proof is the calculation of the Galois group of $f(-T^2)$, where $f$ is a polynomial of degree $n$ with a few variable coefficients: The Galois group is the hyperoctahedral group of order $2^nn!$. 
\end{abstract}

\section{Introduction}
An integer $n$ is  a sum of two squares if there exist  $a,b\in \ZZ$ such that $n=a^2+b^2$. 
Fermat's theorem characterizes sums of two squares as those integers for which in their prime factorization 
each prime $p\equiv 3\pmod 4$ appears with even multiplicity. This can be deduced by studying the prime factorization in the ring of Gaussian integers $\ZZ[i]$ and noting that $n$ is a sum of two squares if and only if it is a norm of an element from $\ZZ[i]$. We let 
\begin{equation}
b(n) = \begin{cases}
1, & n = a^2+b^2 \\
0, & \mbox{otherwise}
\end{cases}
\end{equation}
be the characteristic function of the set of integers that are a sum of two squares. 

\subsection{Landau's Theorem}
A famous theorem of Landau \cite{Landau} gives the mean value of $b(n)$:
\begin{equation}\label{LandauThm}
\left<b(n)\right>_{n\leq x} \;:=\; \frac{1}{x} \sum_{n\leq x} b(n) \;\sim\; K \frac{1}{\sqrt{\log x}}, \qquad x\to \infty
\end{equation}
where
\begin{equation}
K\;=\;\frac{1}{\sqrt{2}}\prod_{p\, \equiv\, 3\,(\mathrm{mod}\,4)}(1-p^{-2})^{-1/2}\;\approx\; 0.764  
\end{equation}
is the Landau-Ramanujan constant. 
The reader may note the similarity of \eqref{LandauThm} to the Prime Number Theorem that gives the mean value of the characteristic function of the primes $\lambda$:
\[
\left<\lambda(n)\right>_{n\leq x} \sim \frac{1}{\log x}.
\]
Indeed, \eqref{LandauThm} is based on 
Fermat's theorem,
which allows one to express the generating function  $\sum_{n=1}^{\infty}b(n)n^{-s}$ in terms of the Riemann zeta function
and the Dirichlet $L$-function formed with the non-principal character modulo $4$.

\subsection{Sums of Two Squares in Short Intervals}

By (\ref{LandauThm}), the average gap between two consecutive sums of two squares is about $K^{-1}\sqrt{\log x}$,
hence \emph{naively}, one would expect that if
\begin{equation}\label{eq:rangeofparameters}
\lim\limits_{x\to \infty}\frac{\phi(x)}{\sqrt{\log x}}=\infty \qquad \mbox{and} \qquad \phi(x)<x,
\end{equation} 
then the mean value of $b(n)$ in the interval $\{n\in\mathbb{Z}:|n-x|\leq\phi(x)\}$ is
\begin{equation}\label{eq:sos_shortintervals}
\left<b(n)\right>_{|n-x|\leq \phi(x)} \sim K\frac{1}{\sqrt{\log x}}, \qquad x\to \infty.
\end{equation}
The problem of estimating the mean value of $b(n)$ in such intervals
has a long history. 

When restricting to all $x$ but a set of asymptotic density $0$, we have the correct upper and lower bounds, up to constants: See
Friedlander \cite{Friedlander1,Friedlander2} and Hooley \cite{Hooley4} for upper bounds;   Plaskin \cite{Plaskin}, Harman \cite{Harman}, and Hooley \cite{Hooley4} for lower bounds. See Iwaniec \cite{Iwaniec} for the application of the half dimensional sieve to this problem and the exposition \cite[\S14.3]{FI}. 

For all $x$, we have a Maier type phenomenon: Balog and Wooley \cite{BalogWooley} show that for $\phi(x) = (\log x)^A$, $A>\frac{1}{2}$, there exist sequences $x^{+}_k$ and $x^{-}_k$ tending to $\infty$ such that $\left<b(n)\right>_{|n-x_k^{\pm}|\leq \phi(x_k^{\pm})}$ is asymptotically bigger/smaller than what is expected by \eqref{eq:sos_shortintervals}. Thus, \eqref{eq:sos_shortintervals} cannot be taken so naively, and one must restrict the range \eqref{eq:rangeofparameters}.  

One natural restriction is to $\phi(x)=x^{\epsilon}$ with fixed $0<\epsilon <1$. It is  a folklore conjecture that \eqref{eq:sos_shortintervals} should hold; i.e., that for any fixed $0<\epsilon <1$: 
\begin{equation}\label{eq:sos_shortintervals_conj}
\left< b(n)\right>_{|n-x|\leq x^{\epsilon}}\sim K\frac{1}{\sqrt{\log x}}, \qquad x\to \infty. 
\end{equation}
Using methods of Ingham, Montgomery, and Huxley for primes, one can confirm this conjecture for $\epsilon>\frac{7}{12}$ unconditionally and for $\epsilon>\frac{1}{2}$ assuming the Riemann Hypothesis for both the Riemann zeta function and the Dirichlet $L$-function formed with the non-principal character modulo $4$, see \cite{Hooley3}.

\subsection{Landau Theorem in Function Fields}
The classical analogy between number fields and global function fields
translates problems about the integers into problems for polynomials over finite fields,
see \cite{Rudnick} for 
the classical analogue of the Prime Number Theorem and a survey of some of the recent work in this area.
In this note, we will study a function field analogue of sums of two squares in short intervals.

Let $q$ be an odd prime power and let $\FF_q[T]$ be the ring of polynomials over a finite field $\FF_q$ with $q$ elements. We denote by  $\M_{n,q}\subseteq \FF_q[T]$ the subset of monic polynomials of degree $n$. 
Following \cite{BSW},
the analogue of a sum of two squares that we will consider in this study is a polynomial of the form 
\[
f=A^2 + T B^2, \qquad A,B\in \FF_q[T].
\]
In other words, we consider norms from the ring $\FF_q[\sqrt{- T}]$, which we take as the analogue of $\ZZ[i]$. 
(We could as well study polynomials of the form $f=A^2-\alpha TB^2$ with a fixed $\alpha\in\mathbb{F}_q^\times$,
but in order to keep the presentation simple, we restrict to $\alpha=-1$.)
We define for $f\in \M_{n,q}$:
\[
b_q(f) = 
\begin{cases}
1,& f =A^2 + TB^2\\
0,& \mbox{otherwise.}
\end{cases}
\]
The analogue of Landau's theorem \eqref{LandauThm} in function fields should give the asymptotic of the mean value 
$$
 \left<b_q(f)\right>_{f\in \M_{n,q}} := \frac{1}{\#\mathcal{M}_{n,q}}\sum_{f\in\mathcal{M}_{n,q}}b_q(f)
$$
as $q^n\to \infty$. 
We note that $q^n$ has several ways to tend to infinity and the asymptotic value is different in different limits, see \cite{BSW}.
In this work we will be interested in the range of parameters when $q$ is much larger than $n$. In this limit, 
a consequence of a result of the second author, Smilansky, and Wolf \cite[Thm.~1.2]{BSW}, says that
\begin{equation}\label{eqSWL}
\left< b_q(f) \right>_{f\in \M_{n,q}} = \frac{1}{4^n}\binom{2n}{n} + O_n(q^{-1}),
\end{equation}
where the implied constant depends only on $n$

\subsection{Sums of Two Squares in Short Intervals in $\FF_q[T]$}
On $\FF_q[T]$ we have the norm function 
\[
\|h\| = q^{\deg h} \quad \mbox{and} \quad \|0\|=0.
\]
Thus, following \cite{KeatingRudnick}, for $0<\epsilon<1$ and $f_0\in\mathcal{M}_{n,q}$, we consider 
\[
\{ f\in \FF_q[T]: \|f-f_0\| \leq \|f_0\|^{\epsilon}\} = \{f_0+h: h\in\mathbb{F}_q[T],\deg h\leq \epsilon \deg f_0\}
\]
as the analogue of $\{n\in\mathbb{Z}:|n-x|\leq x^{\epsilon}\}$ in \eqref{eq:sos_shortintervals_conj}.
Our main result in this work is a function field analogue of \eqref{eq:sos_shortintervals_conj} in the limit $q\rightarrow\infty$:
\begin{theorem}\label{maintheorem}
For odd $q$, $n>2$, $1>\epsilon \geq \frac{2}{n}$, and $f_0\in \M_{n,q}$ we have 
\begin{equation}\label{eqnmaintheorem}
\left<b_q( f) \right>_{\|f-f_0\|\leq\|f_0\|^{\epsilon}} = \frac{1}{4^n}\binom{2n}{n} + O_n(q^{-1/2}),
\end{equation}
where 
the implied constant depends only on $n$.
\end{theorem}

Note that the error term in (\ref{eqSWL}) is smaller than in (\ref{eqnmaintheorem}).
However, the method from \cite{BSW} fails here.
For $\epsilon<\frac{2}{n}$, (\ref{eqnmaintheorem}) no longer holds, as we show in Section~\ref{sec:smallEpsilon}.

\subsection{Methods}

Our approach is based on the function field analogue of Fermat's theorem \cite[Thm.~2.5]{BSW}:
\begin{theorem}\label{thm:representable}
Let $f\in\mathcal{M}_{n,q}$. Then $b_q(f)=1$ if and only if in the prime factorization of $f$,
every prime polynomial $P\in\mathbb{F}_q[T]$ with $P(-T^2)\in\mathbb{F}_q[T]$ irreducible
appears with even multiplicity.
\end{theorem}
In Section~\ref{sec:Frobenius},
we take a `generic' polynomial for the problem,
\[
f_{(A_i)}(T) = f_0 + \sum_{0\leq i\leq \epsilon n} A_iT^i,
\] 
with the $A_i$ variables. 
We use Theorem~\ref{thm:representable} and Galois theory to formulate the property that, under a specialization $(A_i)\mapsto (a_i)$ of the variable coefficients to elements of $\FF_q$, $b_q(f_{(a_i)})=1$,  in terms of the Frobenius element. 
This, based on an explicit Chebotarev theorem, reduces the proof of Theorem \ref{maintheorem} to a calculation of the Galois group of 
$f_{(A_i)}(-T^2)$, which we undertake in Section~\ref{calculation} --
it turns out to be the hyperoctahedral group of order $2^nn!$ (cf.~Section \ref{sec:hyperoctahedral}),
also known as the Coxeter group of type $B_n$, the group of symmetries of the $n$-dimensional hypercube.

\section{The hyperoctahedral group}
\label{sec:hyperoctahedral}

We keep in this section to our setting and do not work in full generality to make the exposition as simple as possible. 

\begin{definition}
Recall that a group $G$ acting on a set $\Omega$ is called a {\em permutation group} if the corresponding map $G\to {\rm Sym}(\Omega)$ is injective (i.e.\ no nontrivial element of $G$ acts trivially on $\Omega$). The regular action of $G$ on itself (i.e.\ via multiplication) always makes $G$ a permutation group.  
\end{definition}

\begin{definition}
Let $G$ be a permutation group on $\Omega$ (with left action), let $C_2=\{\pm1\}$ be the cyclic group of order two, 
and let 
$$
 C_2^{\Omega}:= \{ \xi\colon \Omega \to C_2\}
$$ 
be the 
group of functions from $\Omega$ to $C_2$.
Then $G$ acts (from the right) on $C_2^{\Omega}$ by 
\[
\xi^\sigma(\omega) = \xi(\sigma.\omega), \qquad \sigma\in G, \ \omega\in \Omega.
\]
The corresponding semidirect product 
\[
C_2\wr G := C_2^\Omega\rtimes G
\]
is called the (permutational) {\em wreath product} of $C_2$ and $G$. Its action on  $C_2\times \Omega$ via
\[
(\xi,\sigma).(x,\omega) = (\xi(\sigma.\omega)x,\sigma.\omega), \qquad \xi\in C_2^\Omega, \ \sigma\in G, \ x\in C_2,\ \omega\in\Omega
\]
makes it a permutation group.
In the special case where $G=S_n$ is the symmetric group acting on $[n]:=\{1,\dots,n\}$, 
the group $C_2\wr S_n$ is also called the \emph{hyperoctahedral group}.
\end{definition}

We introduce a subset $X_n\subseteq C_2\wr S_n$ of the hyperoctahedral group that will play a key role in the study that follows:
\begin{equation}\label{Xn}
X_n = \left\{ (\xi,\pi)\in C_2\wr S_n \;:\; \prod_{\omega \in \Omega'}\xi(\omega)=1\;\;\mbox{ for all orbits $\Omega'\subseteq[n]$ of $\pi$}\right\}.
\end{equation}
We compute
the probability that a randomly chosen element of $C_2\wr S_n$ lies in
$X_n$.
For this, recall that a {\rm partition} $\lambda\vdash n$ of $n$ is a tuple $\lambda=(\lambda_1,\dots,\lambda_n)$ with $\sum_{j=1}^n j\lambda_j=n$.
The {\em cycle type} of a permutation $\pi\in S_n$ is $\lambda(\pi):=(\lambda_1,\dots,\lambda_n)\vdash n$,
where $\lambda_j$ is the number of orbits of $\pi$ of length $j$.
\begin{lemma}
We have 
\begin{equation}\label{lem:Xn}
 \frac{\#X_n}{\#C_2\wr S_n} = \frac{1}{4^n}\binom{2n}{n}. 
\end{equation}

\end{lemma}
\begin{proof}
For each partition $\lambda\vdash n$, the number of $\pi\in S_n$ with cycle type $\lambda$ is 
$$
 \frac{n!}{1^{\lambda_1}\cdots n^{\lambda_n}\cdot\lambda_1!\cdots\lambda_n!}=n!\cdot\prod_{j=1}^n\frac{1}{\lambda_j!j^{\lambda_j}},
$$
see e.g.~\cite[\S14.3]{AS}. 
If $\pi\in S_n$ has cycle type $\lambda$, then out of the $2^n$ many $(\xi,\pi)\in C_2\wr S_n$, 
there are 
$$
 \prod_{j=1}^n 2^{(j-1)\lambda_j}=2^n\cdot\prod_{j=1}^n\frac{1}{2^{\lambda_j}}
$$
many in $X_n$, as each cycle of $\pi$ determines one function value of $\xi$. Thus, 
$$
 \frac{\#X_n}{\#C_2\wr S_n} = 
\frac{1}{n!\cdot 2^n}\cdot \sum_{\lambda\vdash n}\left(n!\cdot\prod_{j=1}^n\frac{1}{\lambda_j!j^{\lambda_j}}\cdot2^n\cdot\prod_{j=1}^n\frac{1}{2^{\lambda_j}}\right)  
=\sum_{\lambda\vdash n} \prod_{j=1}^n \frac{1}{\lambda_j! (2j)^{\lambda_j}}.
$$
By \cite[Equation~3]{KarlinMcGregor} the RHS equals $\frac{1}{4^n}\binom{2n}{n}$, as needed. (Indeed, taking a sum over all partitions in \cite[Equation~3]{KarlinMcGregor} with $\theta=1/2$ one gets on the one hand $1$, and on the other hand $\sum_{\lambda\vdash n} \prod_{j=1}^n \frac{1}{\lambda_j! (2j)^{\lambda_j}}\LARGE/\frac{1}{4^n}\binom{2n}{n}$.)
\end{proof}

\section{
	Connection with Frobenius elements
}
\label{sec:Frobenius}

We now work in the following setting:
Let $K$ be a field of characteristic $\neq 2$
and let $f\in K[T]$ be a separable polynomial of degree $n$ such that $f(0)\neq 0$.
Let $L$ be a splitting field of $f$ and let 
\[
\Omega=\{\omega_1, \ldots, \omega_n\} \subseteq L
\] 
be the set of roots of $f$.
The Galois group $G=\Gal(L|K)$ of $f$ is a permutation group on $\Omega$, 
which gives us an embedding 
\begin{equation}\label{EMGG}
\pi\colon G\rightarrow S_n, \qquad \sigma\mapsto\pi_\sigma
\end{equation} 
that satisfies $\sigma(\omega_i)=\omega_{\pi_\sigma(i)}$
for all $\omega\in G$ and $i\in[n]$.
For each $i$, choose two square roots $\omega_i^{\pm} = \pm \sqrt{-\omega_i}$ 
and let $M=L(\omega_i^{\pm} : i\in[n])$.
We also denote the map ${\rm Gal}(M|K)\rightarrow S_n$, $\sigma\mapsto\pi_{\sigma|_L}$ by $\pi$.

\begin{lemma}\label{lemma:wreathproduct}
The field $M$ is the splitting field of the separable polynomial $f(-T^2)$ and
the homomorphism
\begin{equation}\label{WreathGalois}
 \Theta\colon \Gal(M|K)\to C_2\wr S_n, \qquad \sigma\mapsto (\xi_\sigma,\pi_\sigma),
\end{equation}
where $\xi_\sigma:[n]\rightarrow\{\pm1\}$ is defined by 
$\sigma(\omega_i^+)=\xi_\sigma(\pi_\sigma(i))(\sigma\omega_i)^+$ for all $i$,
equivalently 
\begin{equation}\label{xi:def}
\xi_\sigma(i) = \frac{\sigma((\sigma^{-1}\omega_i)^+)}{\omega_i^+},
\end{equation}
is an embedding.
\end{lemma}

\begin{proof}
The assumptions that $f(0)\neq0$ and that $f$ is separable imply that $f(-T^2)$ is separable.
It is clear that $M$ is the splitting field of $f(-T^2)$.
Direct computation shows that $\Theta$ is a homomorphism, 
see e.g.~\cite[Lemma 3.7]{B2012}.
Clearly, $\Theta$ is injective: If $(\xi_\sigma,\pi_{\sigma})$ is trivial,
then $\sigma\omega_i=\omega_i$ and $\xi_\sigma(\pi_\sigma(i))=1$,
hence $\sigma(\omega_i^+)=\omega_i^+$ for all $i$,
and therefore $\sigma={\rm id}_M$.
\end{proof}

\begin{lemma}\label{lem:Theta}
The following diagram commutes: 
$$
 \xymatrix{
  {\rm Gal}(M|K)\ar@{->}[d]\ar@{->}[r]^\Theta &  C_2\wr S_n\ar@{->}[d] \\
  {\rm Sym}(\{\omega_i^{\pm} : i\in[n] \})\ar@{->}[r]^{\quad\eta} & {\rm Sym}(C_2\times[n]) \\
 }
$$
Here the vertical arrows are the embeddings induced by the permutation action, $\Theta$ is defined in \eqref{WreathGalois}, and 
the isomorphism $\eta$ is induced from the bijection $\beta\colon\{\omega_i^{\pm}:i\in[n]\}\to C_2\times[n]$ given by  $\beta(\omega_i^\pm)=(\pm1,i)$.
\end{lemma}

\begin{proof}
We have 
$$
 \Theta(\sigma).\beta(\omega_i^\pm) = (\xi_\sigma,\pi_{\sigma}).(\pm1,i) = 
 (\pm\xi_\sigma(\pi_{\sigma}(i)),\pi_{\sigma}(i)) =  \beta(\sigma(\omega_i^\pm))
$$
for all $\sigma\in{\rm Gal}(M|K)$ and all $i$, as claimed.
\end{proof}

\begin{lemma}\label{lem:imageoffrob}
Assume that ${\rm Gal}(M|K)=\left<\phi\right>$ is cyclic and that $f$ is irreducible.
Then $f(-T^2)$ is reducible if and only if
$\prod_{i=1}^n\xi_\phi(i)=1$.
\end{lemma}

\begin{proof}
Let $\Theta(\phi)=(\xi,\pi)$.
Since $f$ is irreducible of degree $n$ and ${\rm Gal}(M|K)=\left<\phi\right>$, 
we have that ${\rm Gal}(M|L)=\left<\phi^n\right>$ and $1,\phi,\dots,\phi^{n-1}$ are  representatives of ${\rm Gal}(M|K)/{\rm Gal}(M|L)\cong G$.
Moreover, the fact that $f$ is irreducible implies that 
$\pi$ is an $n$-cycle.

Since $\Gal(M|K)$ is abelian, $L=K(\omega_1)$ and $M=K(\omega_1^+)$. In particular, 
\[
[M:L]= \frac{[M:K]}{[L:K]}\leq \frac{2\deg f}{\deg f}=2,
\]
and equality holds if and only if $f(-T^2)$ is irreducible. Hence,
\begin{eqnarray*}
f(-T^2)\mbox{ is reducible } &\Longleftrightarrow& M=L  \\
 &\Longleftrightarrow& \omega_1^+\in L\\
&\Longleftrightarrow& \phi^n(\omega_1^+)=\omega_1^+\\
&\Longleftrightarrow& \xi_{\phi^n}(\pi_{\phi^n}(1))=1.
\end{eqnarray*}
Here, the third line follows by Galois correspondence.
Since $\Theta$ is a homomorphism, 
$$
 (\xi_{\phi^n},\pi_{\phi^n})=\Theta(\phi^n)=\Theta(\phi)^n=(\xi,\pi)^n=(\xi\xi^\pi\dots\xi^{\pi^{n-1}},\pi^n)=(\xi\xi^\pi\dots\xi^{\pi^{n-1}},1),
$$
so 
$$
 \xi_{\phi^n}(\pi_{\phi^n}(1))=\prod_{k=0}^{n-1}\xi^{\pi^k}(1)=\prod_{k=0}^{n-1}\xi(\pi^k(1))=\prod_{i=1}^n\xi(i),
$$
where the last equality follows since $\pi$ is an $n$-cycle.
Hence, $f(-T^2)$ is reducible if and only if $\prod_{i=1}^n\xi(i)=1$. 
\end{proof}

The assumption that ${\rm Gal}(M|K)$ is cyclic is satisfied for example
when $K=\mathbb{F}_q$ is a finite field:
In that case, ${\rm Gal}(M|K)$ is generated by the $q$-Frobenius
$\phi_q(x)=x^q$.

\begin{proposition}\label{prop:representable}
Let $q$ be an odd prime power and let $K=\mathbb{F}_q$.
Let $f\in K[T]$ be a separable monic polynomial of degree $n$ with $f(0)\neq 0$,
and let $\Theta$ be as in \eqref{WreathGalois}.  
Then $b_q(f)=1$ 
if and only if $\Theta(\phi_q)\in X_n$.
\end{proposition}

\begin{proof}
Write $\Theta(\phi_q)=(\xi,\pi)$ and
let $f=P_1\cdots P_r$ be the prime factorization of $f$. Since $f$ is separable, i.e.~all the $P_i$'s are distinct, Theorem \ref{thm:representable} asserts that $b_q(f)=1$ if and only if
$P_i(-T^2)$ is reducible for all $i$.
The set $\Omega=\{\omega_1,\dots,\omega_n\}$ of roots of $f$ is partitioned as $\Omega=\coprod_{i=1}^r\Omega_i$,
where $\Omega_i=\{\omega_{k_{i1}},\dots,\omega_{k_{in_i}}\}$ is the set of roots of $P_i$.
As each $P_i$ is irreducible, the sets $\{k_{i1},\dots,k_{in_i}\}$ for $i=1,\dots,r$ are exactly the orbits of $\pi$.

By Lemma \ref{lem:imageoffrob},
$P_i(-T^2)$ is reducible if and only if $\prod_{j=1}^{n_i}\xi_i(j)=1$,
where $(\xi_i,\pi_i)=\Theta_i(\phi_q)$ with $\Theta_i$ as in \eqref{WreathGalois} for $P_i(-T^2)$, that is to say,
$$
 \Theta_i\colon{\rm Gal}(M_i|K)\rightarrow C_2\wr S_{n_i},
$$
with 
$M_i$ the splitting field of $P_i(-T^2)$.
However, by \eqref{xi:def}, we have 
$$
 \xi_i(j) = \frac{\phi_q((\phi_q^{-1}\omega_{k_{ij}})^+)}{\omega_{k_{ij}}^+} = \xi(k_{ij})\quad\mbox{ for }j=1,\dots,n_i, 
$$ 
so we see that $\prod_{j=1}^{n_i}\xi_i(j)=\prod_{j=1}^{n_i}\xi(k_{ij})$ is the product over the orbit $\{k_{i1},\dots,k_{in_i}\}$ of $\pi$. 
We conclude that $P_i(-T^2)$ is reducible for all $i$ if and only if $(\xi,\pi)\in X_n$.
\end{proof}

\section{The generic Galois group}\label{calculation}

In this section we compute the Galois group of a suitable generic polynomial.

\begin{definition}
Let $K$ be a field.
We say that $x_1,\dots,x_n\in K^\times$ are {\em square-independent} if their residues in $K^\times/(K^\times)^2$ are $\mathbb{F}_2$-linearly independent,
i.e.\ if the subspace $V\subseteq\mathbb{F}_2^n$ consisting of those $\epsilon=(\epsilon_1,\dots,\epsilon_n)\in\mathbb{F}_2^n$ with 
$$
 \prod_{i=1}^nx_i^{\epsilon_i}\in K^{\times2}
$$ 
is trivial.
Denote 
$$
 w(\epsilon):=\#\{i:\epsilon_i\neq0\}.
$$ 
\end{definition}

The following general lemma is well-known:

\begin{lemma}\label{lem:invariantsubspace}
For $n\in\mathbb{N}$,
consider the standard representation of $S_n$ on $\mathbb{F}_2^n$.
The only invariant subspaces $V\subseteq\mathbb{F}_2^n$ are the following:
\begin{enumerate}
\item $V_0=\{(0,\dots,0)\}$
\item $V_1=\{(0,\dots,0),(1,\dots,1)\}$
\item $V_{n-1}=\{\epsilon\in\mathbb{F}_2^n:w(\epsilon)\equiv 0 \;(\mathrm{mod}\;2)\}$
\item $V_n=\mathbb{F}_2^n$
\end{enumerate}
\end{lemma}

\begin{proof}
If an invariant subspace $V\subseteq\mathbb{F}_2^n$ is different from $V_0$ and $V_1$,
then there exists $0\neq\epsilon\in V$ with $w(\epsilon)<n$.
Applying a suitable transposition $\sigma\in S_n$, we get some $\epsilon'=\epsilon+\sigma\epsilon\in V$ with $w(\epsilon')=2$.
This immediately implies that $V_{n-1}\subseteq V$, but $V_n/V_{n-1}\cong\mathbb{F}_2$, so either $V=V_{n-1}$ or $V=V_n$.
\end{proof}

\begin{lemma}\label{lemma:sqindepwreathproduct}
Let $K$ be a field with ${\rm char}(K)\neq 2$ 
and $f(T)\in K[T]$ a monic separable polynomial of degree $n$ with $f(0)\neq 0$.
Let $G={\rm Gal}(f(T)|K)$ and let $\pi:G\lra S_n$ be the embedding $\sigma \mapsto \pi_{\sigma}$ defined in \eqref{EMGG}. 
Assume that the image $\pi(G)$ in $S_n$ has only $V_0,V_1,V_{n-1},V_n\subseteq\mathbb{F}_2^n$ as invariant subspaces.
Write $f(T)=\prod_{i=1}^n(T+y_i)$ and let $L=K(y_1,\dots,y_n)$ be the splitting field of $f$.
If $f(0)$ and $y_1$ are square-independent in $L$, then
${\rm Gal}(f(-T^2)|K)\cong C_2\wr G$.
\end{lemma}

\begin{proof}
By assumption, $f(0)=y_1\cdots y_n$ and $y_1$ are square-independent in $L$.
In particular, $(1,\dots,1)$ and $(0,1,\dots,1)$ do not lie in the
subspace $V\subseteq\mathbb{F}_2^n$ consisting of those
$\epsilon\in\mathbb{F}_2^n$ with
$\prod_{i=1}^n y_i^{\epsilon_i}\in L^{\times2}$,
which is $\pi(G)$-invariant by assumption.
Therefore, $V=V_0$,
proving that $y_1,\dots,y_n$ are square-independent in $L$.

Hence, by Kummer theory (cf.~\cite[Ch.~VI Thm.~8.1]{Lang}), if $M:=K(\sqrt{y_1},\dots,\sqrt{y_n})$ denotes the splitting field of $f(-T^2)$,
then $[M:L]=2^n$.
The image $H$ of the embedding $\Theta\colon \Gal(M|K)\to C_2\wr S_n$ of Lemma \ref{lemma:wreathproduct} satisfies $H\leq C_2\wr \pi(G)$.
Therefore, 
$$
 \#{\rm Gal}(f(-T^2)|K)=[M:L]\cdot[L:K]=2^n\cdot|G|=\#(C_2\wr G).
$$ 
We conclude that ${\rm Gal}(f(-T^2)|K)\cong C_2\wr G$.
\end{proof}

\begin{lemma}\label{lem:squareindependent}
Let $K$ be a field with ${\rm char}(K)\neq 2$ 
and $f(T)\in K[T]$ a monic polynomial of degree $n$ with ${\rm Gal}(f(T)|K)\cong S_n$.
Write $f(T)=\prod_{i=1}^n(T+y_i)$ and let $L=K(y_1,\dots,y_n)$ be the splitting field of $f$.
Assume that $f(0)$ and ${\rm discr}(f)$ are square-independent in $K$,
and that $f(0)$ and $y_1$ are square-independent in $K(y_1)$.
Then $f(0)$ and $y_1$ are square-independent in $L$.
\end{lemma}

\begin{proof}
Let $K_1=K(y_1)$ and $f_1(T)=f(T)/(T+y_1)\in K_1[T]$.
Let $x=f(0)$ and $y=y_1$,
and suppose that 
$x^ay^b\in L^{\times2}$ 
with $a,b\in\{0,1\}$ and either $a=1$ or $b=1$.
We identify ${\rm Gal}(L|K)$ with $S_n$ via the map $\pi$ given in \eqref{EMGG}.
Since ${\rm Gal}(L|K_1)$ is the stabilizer of $y$, it is isomorphic to $S_{n-1}$.
Therefore, the fixed field $L_1=K_1(\sqrt{{\rm discr}(f_1)})$ of the alternating group $A_{n-1}$ is the unique quadratic extension of $K_1$ inside $L$
(cf.~\cite[Corollary 4.2]{Milne}).
So, since $x^ay^b\notin K_1^{\times2}$ by assumption, we conclude that $K_1(\sqrt{x^ay^b})=L_1$,
or, in other words, 
$$
 x^ay^b{\rm discr}(f_1)\in K_1^{\times2}.
$$ 
Taking the norm ${\rm N}={\rm N}_{K_1|K}$ in the extension $K_1|K$, we get that ${\rm N}(x^ay^b{\rm discr}(f_1))\in K^{\times2}$.
Observe that ${\rm N}(x)=x^n$, ${\rm N}(y)=y_1\cdots y_n=x$, and ${\rm N}({\rm discr}(f_1))={\rm discr}(f)^{n-2}$:
Indeed, if we take as representatives for $S_n/S_{n-1}$ the transpositions $\tau_k=(1 \; k)$ for $k=1,\dots,n$,
then
$$
 {\rm N}({\rm discr}(f_1))=\prod_{k=1}^n {\rm discr}(f_1)^{\tau_k}=
\prod_{k=1}^n\prod_{2\leq i<j\leq n}(y_{\tau_k(i)}-y_{\tau_k(j)})^2,
$$
and each factor $(y_i-y_j)^2$ with $1\leq i<j\leq n$ occurs $n-2$ times, namely once for each $k\notin\{i,j\}$.
Together, we conclude that 
$$
 {\rm N}(x^ay^b{\rm discr}(f_1))=x^{an+b}{\rm discr}(f)^{n-2}\in K^{\times2}.
$$
If $n-2$ is odd, then this immediately contradicts the assumption that $x$ and ${\rm discr}(f)$ are square-independent in $K$.
Similarly, if $an+b$ is odd.
If both $n-2$ and $an+b$ are even, then $b=0$ and thus $a=1$, so $x\in L^{\times2}$,
hence $K(\sqrt{x})$ is the unique quadratic extension of $K$ inside $L$, namely the fixed field
$K(\sqrt{{\rm discr}(f)})$ of $A_n$, contradicting again the assumption that $x$ and ${\rm discr}(f)$ are square-independent in $K$.
\end{proof}

\begin{lemma}\label{lemma:generalLemma}
Let $\tilde{f}(T)\in K[T]$ be a separable polynomial and let $f(T)=\tilde{f}(T)+A\in K(A)[T]$ where $A$ is transcendental over $K(T)$. 
Then ${\rm discr}(f)\in K[A]$ is not divisible by $A$.
\end{lemma}

\begin{proof}
Consider $g(A)={\rm discr}(f)\in K[A]$.
Since ${\rm discr}(f)$ is a polynomial in the coefficients of $f$, we have $g(a)={\rm discr}(\tilde{f}+a)$ for every $a\in K$.
In particular, $g(0)={\rm discr}(\tilde{f})\neq 0$ since $\tilde{f}$ is separable, so $A$ does not divide $g$.
\end{proof}

\begin{proposition}\label{prop:genericGal}
Let $F$ be a field of characteristic different from $2$, 
let $n>m\geq 2$ be integers and let $f_0\in F[T]$ be a monic polynomial of degree $n$.
Define $K=F(A_0,\dots,A_m)$, where $A_0,\dots,A_m$ are independent variables.
Then the polynomial
$$
 f(T) \;=\; f_0(T) + \sum_{i=0}^m A_iT^i \;\in K[T]
$$
satisfies
$$
 {\rm Gal}(f(T)|K)\cong S_n \quad\mbox{ and }
\quad{\rm Gal}(f(-T^2)|K)\cong C_2\wr S_n.
$$ 
\end{proposition}

\begin{proof}
Write $f_0=T^n+\sum_{i=0}^{n-1} a_iT^i$.
By replacing $A_i$ by $A_i-a_i$, 
we may assume without loss of generality that $a_i=0$ for $0\leq i\leq m$.
In particular, $f(0)=A_0$.
Applying \cite[Proposition 3.6]{PrimePoly2014} with $k=n$ and $g=1$ gives that ${\rm Gal}(f(T)|K)\cong S_n$.
In particular, ${\rm discr}(f)\notin K^{\times2}$.  
Write $f(T)=\prod_{i=1}^n(T+y_i)$.
We will now verify the assumptions of Lemma~\ref{lem:squareindependent}.\\

\noindent
{\em Claim 1: $f(0)$ and ${\rm discr}(f)$ are square-independent in $K$}
\nopagebreak

Let $\tilde{f}(T)=f(T)-A_0\in K_0[T]$, where $K_0=F(A_1,\dots,A_m)$, and 
$$
 g(T)=T^{-1}\cdot\tilde{f}(T)=T^{n-1}+a_{n-1}T^{n-2}+\cdots +A_2T+A_1\in K_0[T].
$$ 
Since $g$ is monic and linear in $A_1$, it is irreducible in $K_0[T]$ by Gauss' lemma.
Therefore, since $g(0)=A_1\neq 0$ and $g'(0)=A_2\neq 0$, both $g(T)$ and $\tilde{f}(T)=Tg(T)$ are separable.
Thus, by Lemma \ref{lemma:generalLemma}, ${\rm discr}(f)\in K_0[A_0]$ is not divisible by $A_0$.
In particular, 
$$
 A_0\cdot{\rm discr}(f)\notin K^{\times2}.
$$ 
Together with ${\rm discr}(f)\notin K^{\times2}$ and the obvious fact that $A_0\notin K^{\times2}$,
we conclude that $A_0$ and ${\rm discr}(f)$ are square-independent in $K$.\\

\noindent
{\em Claim 2: $f(0)$ and $y_1$ are square-independent in $K(y_1)$}
\nopagebreak

>From $f(-y_1)=0$ we see that
$$
 A_1= -y_1^{-1}\cdot\left((-y_1)^n+\sum_{i=m+1}^{n-1} a_i(-y_1)^i + \sum_{i=2}^m A_i(-y_1)^i + A_0\right)\in K_1(y_1),
$$
where $K_1=F(A_0,A_2,\dots,A_m)$.
Thus, $K(y_1)=K_1(y_1)=F(A_0,A_2,\dots,A_m,y_1)$, which, since ${\rm tr.deg}(K(y_1)|F)=m+1$, implies that
$A_0$, $A_2,\dots,A_m$ and $y_1$ are algebraically independent over $F$
(in other words, the $(m+1)$-dimensional hypersurface defined by $f=0$ is rational). 
In particular, $A_0$ and $y_1$ are square-independent in $K(y_1)$.\\

\noindent
{\em Conclusion of the proof}:
\nopagebreak

Using Claim 1 and Claim 2, we can now apply Lemma \ref{lem:squareindependent} 
and conclude that $f(0)$ and $y_1$ are square-independent in 
the splitting field of $f(T)$.
Therefore, since $S_n$ has no invariant subspaces other than the ones of Lemma \ref{lem:invariantsubspace}, we may invoke Lemma \ref{lemma:sqindepwreathproduct} and get that ${\rm Gal}(f(-T^2)|K)\cong C_2\wr S_n$.
\end{proof}

\section{Proof of Theorem \ref{maintheorem}}
The proof of Theorem \ref{maintheorem} follows the pattern of similar proofs in the literature, like in \cite{ABR,BB,PrimePoly2014,Entin}.
The main ingredient is an explicit Chebotarev theorem, which we recall now.

Fix $r,d\in\mathbb{N}$ and let $q$ be a prime power. 
We let $\mathbf{A}=(A_1,\dots,A_d)$ be a $d$-tuple of variables
and define $R=\mathbb{F}_q[\mathbf{A}]$ and $K=\mathbb{F}_q(\mathbf{A})$.
For a monic separable polynomial $g\in R[T]$ of degree $r$,
we write 
$$
 g(T)=\prod_{i=1}^r(T-\rho_i)
$$ 
and let $M=K(\rho_1,\dots,\rho_r)$ be a splitting field of $g$.
\emph{We assume that $M$ is regular over $\FF_q$, i.e.~$M\cap \overline{\FF}_{q}=\FF_q$, where $\overline{\FF}_q$ is an algebraic closure of $\FF_q$.}
The action of ${\rm Gal}(M|K)$ on $\{\rho_1,\dots,\rho_r\}$ induces an embedding 
$$
 \iota\colon {\rm Gal}(M|K)\rightarrow S_r. 
$$ 

For each $\mathbf{a}=(a_1,\dots,a_d)\in\mathbb{F}_q^d$ 
we have the homomorphism $\Phi_\mathbf{a}:R\rightarrow\mathbb{F}_q$ given by $\Phi_\mathbf{a}(A_i)=a_i$ for all $i$.
For those $\mathbf{a}\in\mathbb{F}_q^d$ which are not a zero of $\Delta:={\rm discr}(g)\in R$,
we can choose an extension of $\Phi_\mathbf{a}$ to a homomorphism
\begin{equation}\label{Phiprime}
\Phi_\mathbf{a}' \colon R[\Delta^{-1},\bfrho]\to \overline{\FF}_q.
\end{equation}
We apply $\Phi_\mathbf{a}$ to polynomials by applying it to their coefficients.
Then 
$$
 g_\mathbf{a}:=\Phi_\mathbf{a}(g)=\prod_{i=1}^r(T-\Phi_\mathbf{a}'(\rho_i))\in\mathbb{F}_q[T],
$$ 
so if $M_\mathbf{a}$ denotes the splitting field of $g_\mathbf{a}$ over $\mathbb{F}_q$,
then the action of 
${\rm Gal}(M_\mathbf{a}|\mathbb{F}_q)$ on the set
$\{\Phi_\mathbf{a}'(\rho_1),\dots,\Phi_\mathbf{a}'(\rho_r)\}$ 
of roots of $g_\mathbf{a}$ 
(which has again $r$ elements since $\Delta(\mathbf{a})\neq0$)
induces an embedding
$$
 \iota_\mathbf{a}:{\rm Gal}(M_\mathbf{a}|\mathbb{F}_q)\rightarrow S_r.
$$ 
As before we denote by $\phi_q\in{\rm Gal}(M_\mathbf{a}|\mathbb{F}_q)$ the $q$-Frobenius.

\begin{theorem}\label{cor:cheb}
There exists a constant $c$ depending only on $d$ and the total degree of $g$ (as a polynomial in $A_1,\dots,A_d,T$)
such that for every $X\subseteq{\rm Gal}(M|K)$ invariant under conjugation,
\[
\left|\# \{ \mathbf{a}\in \FF_q^d : \Delta(\mathbf{a})\neq 0 \mbox{ and }\iota_\mathbf{a}(\phi_q)\in\iota(X) \} -\frac{\#X}{\#{\rm Gal}(M|K)}\cdot
q^d\right|\leq c q^{d-1/2}.
\]
\end{theorem}

\begin{proof}
This is classical. In this form of uniformity it can be deduced immediately from 
\cite[Theorem~A.4]{ABR}.
\end{proof}

\begin{proof}[Proof of Theorem \ref{maintheorem}]
Let $q$ be an odd prime power, $n>2$, $1>\epsilon \geq \frac{2}{n}$, $f_0\in \M_{n,q}$, and put $m=\lfloor \epsilon n\rfloor\geq 2$. 
We let $\mathbf{A}=(A_0,\dots,A_m)$ be a tuple of independent variables and define $K=\mathbb{F}_q(\mathbf{A})$.
Let
$$
 f(T) \;=\; f_0(T) + \sum_{i=0}^m A_iT^i \;\in K[T]
$$
and
$$
 g(T) = (-1)^{n}\cdot f(-T^2) \in K[T].
$$
Now let $L$ be the splitting field of $f$ over $K$, write $f=\prod_{i=1}^n(T-\omega_i)$
and let $\Omega=\{\omega_1,\dots,\omega_n\}\subseteq L$.
For each $i=1,\dots,n$ choose a square root $\rho_i=\sqrt{-\omega_i}$
and let $\rho_{n+i}=-\rho_i$.
Then 
$$
 g(T)=\prod_{i=1}^{2n}(T-\rho_i)
$$ 
and $M=K(\bfrho)$ is the splitting field of $g$. 
Let $\Theta\colon {\rm Gal}(M|K)\rightarrow C_2\wr S_n$ be the homomorphism given in \eqref{WreathGalois}.
By Proposition \ref{prop:genericGal}, 
${\rm Gal}(L|K)\cong S_n$ and 
$\Theta$ is an isomorphism.
As Proposition~\ref{prop:genericGal} also applies to $F=\overline{\mathbb{F}}_q$ instead of $F=\mathbb{F}_q$,
we get that ${\rm Gal}(M\overline{\mathbb{F}}_q|K\overline{\mathbb{F}}_q)={\rm Gal}(M|K)$,
and therefore $M|\mathbb{F}_q$ is regular.

The discriminant $\Delta:={\rm discr}(g)$ is a non-zero polynomial in $\textbf{A}$ of degree $\leq 4n$ (by the resultant formula). 
Therefore,
\begin{equation}\label{eq:zerosofDelta}
\#\{\mathbf{a}\in \mathbb{F}_q^{m+1}: \Delta(\mathbf{a})=0\} \leq 4n q^{m},
\end{equation}
see e.g.~\cite[Ch.~4 Lemma 3A]{Schmidt}.

For $\mathbf{a}\in\mathbb{F}_q^{m+1}$ which is not a zero of $\Delta$
we choose a homomorphism $\Phi_\mathbf{a}'$ as in \eqref{Phiprime}
and let $f_\mathbf{a}:=\Phi_\mathbf{a}(f),g_\mathbf{a}:=\Phi_\mathbf{a}(g)\in\mathbb{F}_q[T]$.
Note that 
$$
 f_\mathbf{a}(T)=\prod_{i=1}^n(T-\Phi_\mathbf{a}'(\omega_i))
$$ 
and 
$$
 g_\mathbf{a}(T)=(-1)^n\cdot f_\mathbf{a}(-T^2)=\prod_{i=1}^{2n}(T-\Phi_\mathbf{a}'(\rho_i)),
$$ 
so $\Omega_\mathbf{a}:=\{\Phi_\mathbf{a}'(\omega_1),\dots,\Phi_\mathbf{a}'(\omega_n)\}$ is the set of zeros of $f_\mathbf{a}$,
$L_\mathbf{a}=\mathbb{F}_q(\Phi_\mathbf{a}'(\bfomega))$ is a splitting field of $f_\mathbf{a}(T)$,
and $M_\mathbf{a}=\mathbb{F}_q(\Phi_\mathbf{a}'(\bfrho))$ is a splitting field of $f_\mathbf{a}(-T^2)$.
Let 
$$
 \Theta_\mathbf{a}:{\rm Gal}(M_\mathbf{a}|\mathbb{F}_q)\rightarrow C_2\wr S_n
$$
be as in \eqref{WreathGalois}
and $\iota_\mathbf{a}:{\rm Gal}(M_\mathbf{a}|\mathbb{F}_q)\rightarrow S_{2n}$ as above.
By Lemma \ref{lem:Theta}, the following diagram commutes:
\begin{equation}\label{diagram}
 \xymatrix{
 {\rm Gal}(M|K)\ar@{->}[r]^\Theta\ar@{->}[rd]_\iota & C_2\wr S_n\ar@{->}[d]  & {\rm Gal}(M_\mathbf{a}|\mathbb{F}_q)\ar@{->}[l]_{\Theta_{\mathbf{a}}}\ar@{->}[ld]^{\iota_\mathbf{a}} \\
  &  S_{2n} &  
 }
\end{equation}

Now let $X_n\subseteq C_2\wr S_n$ be as in \eqref{Xn} and define $X:=\Theta^{-1}(X_n)\subseteq{\rm Gal}(M|K)$.
By Proposition~\ref{prop:representable}, $b_q(f_\mathbf{a})=1$ if and only if $\Theta_\mathbf{a}(\phi_q)\in X_n$.
The commutativity of \eqref{diagram} shows that the latter is equivalent to $\iota_\mathbf{a}(\phi_q)\in\iota(X)$.

Therefore, Theorem \ref{cor:cheb} applied to $g$ with $r=2n$ and $d=m+1$, together with \eqref{eq:zerosofDelta}, 
gives a constant $c_n$ depending
only on $m$, $n$ and the total degree of $g$ such that
\begin{equation}\label{eqMF}
\left|\# \left\{ \mathbf{a}\in \FF_q^{m+1} : b_q(f_\mathbf{a})=1 \right\} -\frac{\#X}{\#{\rm Gal}(M|K)}\cdot q^{m+1}\right|\leq c_n q^{m+1/2}.
\end{equation}
Since $m\leq n$ and the total degree of $g$, which equals $2n$, are independent of $q$ and the choice of the polynomial $f_0$ of degree $n$, 
the constant $c_n$ can be chosen to depend only on $n$.
Plugging \eqref{lem:Xn} into \eqref{eqMF} concludes the proof. 
\end{proof}
\section{Small $\epsilon$}\label{sec:smallEpsilon}
In this section we deal with $0<\epsilon<\frac{2}{n}$. These $\epsilon$'s are not covered by Theorem~\ref{maintheorem}. 
We construct sequences of $f_0=f_{0,q_{i}}\in\M_{n,q_i}$ of a fixed arbitrarily large  degree $n$ such that $\left<b_{q_i}( f) \right>_{\|f-f_0\|\leq\|f_0\|^{\epsilon}}$ asymptotically differs from  \eqref{eqnmaintheorem} as $q_i\to \infty$. This  shows that the restriction on $\epsilon$ in Theorem~\ref{maintheorem} is not redundant.

\subsection{First interval: $0<\epsilon< \frac{1}{n}$}
Let $q$ be an odd prime power. We fix $k\geq 1$ and let $n=2k+1$ and $f_0=T^{2k+1}$. 
Then 
\[
\{f\in\mathbb{F}_q[T]:\|f-f_0\|\leq\|f_0\|^\epsilon\} =\{T^{2k+1}+a : a\in \FF_q\}.
\] 
We note that $b_q(T^{2k+1}+a)=1$ if and only if $a$ is a square in $\FF_q$. 
Indeed, if $b_q(T^{2k+1}+a)=1$, then $T^{2k+1}+a = A^2+TB^2$, so $a=A(0)^2$ is a square and  if $a=b^2$ with $b\in\mathbb{F}_q$, then $T^{2k+1}=b^2+T(T^k)^2$,  so $b_q(T^{2k+1}+a)=1$.

There are exactly $\frac{q+1}{2}$ squares in $\FF_q$,  thus 
\[
\left<b_q( f) \right>_{\|f-f_0\|\leq\|f_0\|^{\epsilon}}=
\frac{(q+1)/2}{q}=
\frac{1}{2}+\frac{1}{2q},
\]
which is obviously not compatible with \eqref{eqnmaintheorem}.

\subsection{Second interval: $\frac{1}{n}\leq \epsilon <\frac{2}{n}$}
Fix a prime $p>2$, let $n=p^2$, $\nu\in\mathbb{N}$, $q=p^{2\nu}$, and $f_0=T^{p^2}\in\mathbb{F}_q[T]$. We compute the asymptotic mean value of $b_q(f)$ for $f$ in 
\[
\{f\in\mathbb{F}_q[T]:\|f-f_0\|\leq \|f_0\|^{\epsilon} \}=\{T^{p^2}+a_1T+a_0: a_1,a_0\in \FF_q\}
\]
as $\nu\to \infty$ (and hence also $q=p^{2\nu}\to\infty$).
\begin{theorem}\label{thm:2parameter}
Let
\begin{equation}\label{eq:Cp}
c_p= 
\frac{1}{2^{p^2}p^2(p^2-1)} + \frac{1}{2^{p}p^2}+ \frac{1}{2^{p^2}(p^2-1)}\cdot\sum_{1\neq d\mid p^2-1}2^{(p^2-1)(d-1)/d}\phi(d) ,
\end{equation}
where $\phi(d)$ is the Euler totient function. Then
\begin{equation}\label{eq:SI-AS}
\left<b_q(f)\right>_{\|f-f_0\|\leq \|f_0\|^{\epsilon}}  \sim c_p , \qquad \nu\to \infty.
\end{equation}
\end{theorem}

Bounding the last summand for $d=p^2-1$ gives that
\[
c_{p} \geq \frac{1}{2^{p^2}(p^2-1)}\cdot 2^{(p^2-1)\cdot(p^2-2)/(p^2-1)}\phi(p^2-1)=\frac{1}{4}\cdot\frac{\phi(p^2-1)}{p^2-1}\gg \frac{1}{\log\log p^2}, \qquad p\to \infty,
\]
as $\phi(n)\gg \frac{n}{\log\log n}$ for $n\rightarrow\infty$.
On the other hand,
\[
\frac{1}{4^{p^{2}}}\binom{2p^{2}}{p^{2}} \sim \frac{1}{\sqrt{\pi}p}, \qquad p\to \infty.
\]
Thus, if we pick $p$ sufficiently large, we see that $c_p > \frac{1}{4^{p^{2}}}\binom{2p^{2}}{p^{2}} $, hence \eqref{eq:SI-AS}  is not compatible with  \eqref{eqnmaintheorem}. 

To prove \eqref{eq:SI-AS}, we take the same approach as the one used to obtain \eqref{eqnmaintheorem}, namely applying the explicit Chebotarev Theorem (Theorem \ref{cor:cheb}); 
however, the respective Galois groups are different, which explains the different asymptotic formula.

Let $F|\mathbb{F}_{p^2}$ be a field extension,
$A_0,A_1$ independent variables,
$K=F(A_0,A_1)$ 
and 
$$
 f(T)=T^{p^2}+A_1T+A_0\in K[T].
$$ 
As ${\rm Aut}(\mathbb{F}_{p^2}|F\cap\mathbb{F}_{p^2})$ is trivial,
\cite[Theorem 2]{Uchida1970} gives that 
\begin{equation}\label{group}
G:=\Gal(f|K)\cong {\rm Aff}(\FF_{p^2}),
\end{equation}
the group of affine linear transformations 
\[
\sigma_{a,b}: x\mapsto ax+b, \qquad a\in \FF_{p^2}^\times,\quad b\in \FF_{p^2}
\]
of the affine line $\AA^1(\FF_{p^2})$
(the isomorphism being an isomorphism of permutation groups).
We start by a few group theoretical properties of $G$.
\begin{lemma}\label{lem:invsubspaces_aff}
Consider $G={\rm Aff}(\FF_{p^2})$ acting on $V=\mathbb{F}_2^{p^2}$ via the embedding $G\to S_{p^2}$. Then the $G$-invariant subspaces of $V$ are the same as the $S_{p^2}$-invariant subspaces; that is to say, the spaces $V_0,V_1,V_{p^2-1}, V_{p^2}$ as in Lemma~\ref{lem:invariantsubspace}.
\end{lemma}

\begin{proof}
Let $U$ be a $G$-invariant subspace of $V$.
We want to apply the results of \cite{Klemm}
to $\mathfrak{G}=G$, $n=p^2$, $\Omega=\mathbb{F}_{p^2}\cong \{1,\dots,n\}$ and the field $K=\mathbb{F}_2$.
Note that in the notation used there, $M_1=V_1$, $M^1=V_{n-1}$, and $M=(M_1+M^1)/M_1=(V_1\oplus V_{n-1})/V_1\cong V_{n-1}$, as $\FF_2[G]$-modules.

Since $\mathfrak{G}$ contains the transitive subgroup $\mathfrak{H}:=\mathbb{F}_{p^2}$ 
for which $2$ does not divide the order of the stabilizer $\mathfrak{H}_a=1$ for $a\in\Omega$,
\cite[Hilfssatz 7(b)]{Klemm} gives that 
\begin{equation}\label{eq:subspaces_aff}
U\subseteq V_{n-1}\quad \mbox{or} \quad V_1 \subseteq U.
\end{equation}

Moreover, since $\mathfrak{G}$ is $2$-transitive on $\Omega$, $2\nmid n$,
and the stabilizer $\mathfrak{G}_a=\mathbb{F}_{p^2}^\times$, for $a=0\in\Omega$ contains the subgroup $\tilde{\mathfrak{H}}:=\mathfrak{G}_a$,
which is transitive on $\Omega_a=\Omega\smallsetminus\{a\}$ and satisfies $2\nmid|\tilde{\mathfrak{H}}_b|=1$ for $b\in\Omega_a$,
\cite[Satz 8(b)]{Klemm} gives that $V_{n-1}$ is simple.
Thus, $U\cap V_{n-1}=V_0$ or $U\cap V_{n-1}=V_{n-1}$.

If $U\cap V_{n-1}=V_0$, then $\dim U\leq 1$ and by \eqref{eq:subspaces_aff} we conclude that either $U=V_0$ or $V_1\subseteq U$ and therefore $U=V_1$.
If $U\cap V_{n-1} = V_{n-1}$, then we conclude from $V/V_{n-1}\cong\mathbb{F}_2$ that either $U=V_{n-1}$ or $U=V_n$.
\end{proof}

For an element $\sigma_{a,b}\in G$ we let $\lambda(\sigma_{a,b}) :=(\lambda_1, \ldots, \lambda_{p^2} )\vdash p^2$ be the 
cycle type of $\sigma_{a,b}$.
\begin{lemma}\label{lem:cycletypeaff}
Let $a\in\FF_{p^2}^\times$, $b\in \FF_{p^2}$ and $\sigma=\sigma_{a,b}$. 
\begin{enumerate}
\item[(a)] If $a=1$ and $b=0$, then $\lambda(\sigma) = \lambda^0 := (p^2,0,\ldots,0)$.
\item[(b)] If $a=1$ and $b\neq 0$, then $\lambda(\sigma)=\lambda^{p+} := (0,\ldots, 0,p,0,\ldots,0)$.
\item[(c)] If $a\neq 1$ has multiplicative order $d$, then $\lambda(\sigma) =\lambda^{d\times}:= (1,0,\ldots, 0,\frac{p^2-1}{d}, 0 ,\ldots, 0)$.
\end{enumerate}

\end{lemma}

\begin{proof}
(a) and (b) are trivial. 

Let $a\neq 1$ be of multiplicative order $d$. 
Then $x=0$ is the unique fixed point of $\sigma_{a,0}$. For each $x\neq 0$, the orbit of $x$ is $\{x,ax,\cdots ,a^{d-1} x\}$ and is of length $d$. So we have exactly $\frac{p^2-1}{d}$ orbits of length $d$. This implies that $\lambda(\sigma_{a,0})=\lambda^{d\times}$. Since $\sigma_{a,b}$ is conjugated to $\sigma_{a,0}$, we get that in fact $\lambda(\sigma_{a,b})=\lambda^{d\times}$ for all $b$, as was needed for (c).
\end{proof}

\begin{proposition}\label{prop:Galoisgroup}
Let $p>2$ be prime, $F|\mathbb{F}_{p^2}$ a field extension, $A_0$ and $A_1$ independent variables,
and $K=F(A_0,A_1)$.
Then the polynomial
$$
 f(T)=T^{p^2}+A_1T+A_0
$$
satisfies
$$
\Gal(f(T)|K)\cong{\rm Aff}(\FF_{p^2}) \quad\mbox{and}\quad \Gal(f(-T^2)|K) \cong C_2\wr {\rm Aff}(\FF_{p^2}).
$$
\end{proposition}

\begin{proof}
Write $f(T)=\prod_{i=1}^{p^2}(T+y_i)$, so that $L=K(y_1,...,y_{p^2})$ is a splitting field of $f$.
Let $K_1=K(y_1)$. Since $0=f(-y_1)$, we have 
\begin{equation}\label{eq:A_0intermsofA_1y_1}
A_0=(y_1^{p^2-1}+A_1)y_1.
\end{equation} 
Thus,  $K_1 =F(A_0,A_1,y_1) = F(A_1,y_1)$. Since the transcendence degree of $K_1$ over $F$ is $2$, this implies that $K_1$ is the field of rational functions in $A_1,y_1$ over $F$. 

As $f'(T)=A_1$, we get that
\[
{\rm discr}(f)=\pm \prod_{i=1}^{p^2} \prod_{j\neq i} (y_i-y_j) = \pm \prod_{i=1}^{p^2} f'(y_i) = \pm A_1^{p^2}.
\]
So, as $p^2$ is odd, ${\rm discr}(f)$ is not a square in $K_1$, hence $L_1:=K_1(\sqrt{{\rm discr}(f)})=K_1(\sqrt{\pm A_1})$ is a quadratic extension of $K_1$ that is contained in $L$. 

Since $\Gal(L|K_1)$ is a stabilizer in $G={\rm Aff}(\FF_{p^2})$ of a point $x\in \FF_{p^2}$, which, without loss of generality, we may choose to be $x=0$, we have $\Gal(L|K_1) \cong \FF_{p^2}^\times$. As $\FF_{p^2}^\times$ is cyclic, $K_1$ has a unique quadratic extension inside $L$ which by the previous paragraph is $L_1$.

By Lemmas~\ref{lemma:sqindepwreathproduct} and \ref{lem:invsubspaces_aff}, it suffices to prove that $A_0=f(0)$ and $y_1$ are square-independent in $L$.
Assume on the contrary that $A_0^ay_1^b\in (L^\times)^2$ for some $a,b\in \{0,1\}$ with either $a=1$ or $b=1$. 
Note that since $K_1$ is a rational function field in $A_1,y_1$,  \eqref{eq:A_0intermsofA_1y_1} implies that $A_0,y_1$ are square-independent in $K_1$, so  $A_0^ay_1^b\notin (K_1^\times)^2$. Thus $K_1(\sqrt{A_0^ay_1^b})$ is a quadratic extension of $K_1$ that is contained in $L$, so it must be equal to $L_1$. 
Thus by \eqref{eq:A_0intermsofA_1y_1},
\[
 \pm A_0^ay_1^bA_1 = \pm (y_1^{p^2-1}+A_1)^ay_1^{a+b}A_1 \in (K_1^\times)^2,
\]
which leads to a contradiction, as $K_1$ is a rational function field in $A_1,y_1$, and $A_1+y_1^{p^2-1}$, $y_1$, and $A_1$ are co-prime in $F[A_1,y_1]$.
\end{proof}

\begin{proof}[Proof of Theorem~\ref{thm:2parameter}]
Let $q=p^{2\nu}$, $f(T)=T^{p^2}+A_1T+A_0$ and $G={\rm Aff}(\FF_{p^2})$.
Since by Proposition~\ref{prop:Galoisgroup} the Galois group of $g(T):=f(-T^2)$ is $C_2\wr G$
both over $\mathbb{F}_q(A_0,A_1)$ and over $\overline{\mathbb{F}}_q(A_0,A_1)$, 
the same line of arguments as in the proof of Theorem~\ref{maintheorem} gives that 
\begin{equation}\label{eq:forvsi}
\left<b_q(f)\right>_{\|f-f_0\|\leq \|f_0\|^{\epsilon}}\sim \frac{\#(X_{p^2}\cap C_2\wr G)}{\#(C_2\wr G)},
\end{equation}
as $\nu\to \infty$. By Lemma~\ref{lem:cycletypeaff}, the number $N_{\lambda}$ of elements of $G$ of cycle type $\lambda$ is
\[
N_{\lambda} = \begin{cases}
1, & \lambda = \lambda^0,\\
p^2-1, & \lambda = \lambda^{p+},\\
p^2\phi(d), &\lambda = \lambda^{d\times},\ 1\neq d\mid p^{2}-1,\\
0, &\mbox{otherwise}.
\end{cases}
\]
Therefore, as we saw in the proof of \eqref{lem:Xn}, one has
\begin{equation}\label{eq:XnG}
\begin{split}
\#(X_{p^2}\cap C_2\wr G) &= \sum_{\lambda\vdash p^2}  N_{\lambda} \prod_{j=1}^{p^2} 2^{\lambda_j(j-1)} \\
&=1+(p^2-1)2^{p(p-1)} + p^2\sum_{1\neq d\mid p^2-1}\phi(d)2^{(d-1)(p^2-1)/d}.
\end{split}
\end{equation}
Since $\#(C_2\wr G) = 2^{p^2}p^2(p^2-1)$, by \eqref{eq:XnG} it follows that $\frac{\#(X_{p^2}\cap C_2\wr G)}{\#(C_2\wr G)} = c_p$ (with $c_p$ defined in \eqref{eq:Cp}), and thus by \eqref{eq:forvsi}, the proof is done. 
\end{proof}

\section*{Acknowledgements}
The authors are grateful to Alexei Entin for suggesting to them the characterization of 
sums of squares in terms of the Frobenius,
to Peter M\"uller for pointing them to the paper of Klemm,
and to Ron Peled for introducing the Ewens sampling formula to them.

The first author was partially sponsored by the Shulamit Aloni Grant for promoting women in science of the Israeli Ministry of Science, Technology and Space no.\ 3-11924  and the first and second authors by a grant of the Israel Science Foundation no.\ 952/14. 
The third author was supported by a research grant from the Ministerium f\"ur Wissenschaft, Forschung und Kunst Baden-W\"urttemberg.

\end{document}